\documentclass[11pt]{amsart}

\usepackage{amssymb}
\usepackage{mathrsfs}
\usepackage{amsfonts}
\usepackage{latexsym,amsmath,amsthm,amssymb,amsfonts}
\usepackage[usenames]{color}
\usepackage{xspace,colortbl}
\usepackage{graphicx}

\newcommand{\be}{\begin{equation}}
\newcommand{\ee}{\end{equation}}
\newcommand{\beq}{\begin{eqnarray}}
\newcommand{\eeq}{\end{eqnarray}}

\newtheorem{prop}{Proposition}[section]

\newtheorem{remark}[prop]{Remark}

\def\begeq{\begin{equation}}
\def\endeq{\end{equation}}

\def\tr{{\rm tr}}

\def\odot{\setbox0=\hbox{$\bigcirc$}\relax \mathbin {\hbox
to0pt{\raise.5pt\hbox to\wd0{\hfil $\wedge$\hfil}\hss}\box0 }}

\numberwithin{equation} {section}

\numberwithin{equation}{section}
\newtheorem{theorem}{\bf Theorem}[section]
\newtheorem{proposition}[theorem]{\bf Proposition}

\newtheorem{lemma}[theorem]{\bf Lemma}

\allowdisplaybreaks

\begin{document}

\title[$L_p$-Christoffel-Minkowski problem]
{Uniqueness of solutions to $L_p$-Christoffel-Minkowski problem for $p<1$}

\author{
Li Chen}
\address{
Faculty of Mathematics and Statistics, Hubei Key Laboratory of
Applied Mathematics, Hubei University,  Wuhan 430062, P.R. China}
\email{chenli@hubu.edu.cn}


\date{}
\begin{abstract}
$L_p$-Christoffel-Minkowski problem arises naturally in the
$L_p$-Brunn-Minkowski theory. It connects both curvature measures
and area measures of convex bodies and is a fundamental problem in
convex geometric analysis. Since the lack of Firey's extension of
Brunn-Minkowski inequality and constant rank theorem for $p<1$, the
existence and uniqueness of $L_p$-Brunn-Minkowski problem are
difficult problems. In this paper, we prove a uniqueness theorem for
solutions to $L_p$-Christoffel-Minkowski problem with $p<1$ and
constant prescribed data. Our proof is motivated by the idea of
Brendle-Choi-Daskaspoulos's work on asymptotic behavior of flows by
powers of the Gaussian curvature. One of the highlights of our
arguments is that we introduce a new auxiliary function $Z$ which is
the key to our proof.
\end{abstract}

\maketitle {\it \small{{\bf Keywords}: $L_p$-Christoffel-Minkowski
problem, uniqueness, convex solutions.}

{{\bf MSC}: Primary 35J15, Secondary
35J60.}
}

\section{Introduction}

Convex geometry plays an important role in the development of fully
nonlinear partial differential equations. The classical Minkowski
problem and the Christoffel-Minkowski problem in general, are
beautiful examples of such interactions. The classical core of
convex geometry is the Brunn-Minkowski theory, the Minkowski sum,
the mixed volumes. The $L_p$-Brunn-Minkowski theory is an extension
of the classical Brunn-Minkowski theory. The roots of the
$L_p$-Brunn-Minkowski theory date back to the middle of the
twentieth century, but its active development had to await the
emergence of the concept of $L_p$-surface area measure in
\cite{Lu93} in the early 1990's.

Let $\mathcal{K}_{0}^{n+1}$ denote the class of convex bodies
(compact convex subsets) in $n+1$-dimensional Euclidean space
$\mathbb{R}^{n+1}$ that contain the origin in their interiors and
$h_{K}: \mathbb{S}^n\rightarrow \mathbb{R}$ be the support function
of a compact convex subset $K \in \mathbb{R}^{n+1}$, which
determines $K$ uniquely and is defined by $h_{K}(x)=\max\{x\cdot X:
X \in K\}$ for $x \in \mathbb{S}^n$, where $x\cdot X$ is the
standard inner product of $x$ and $X$ in $\mathbb{R}^{n+1}$. For
each real $p\geq 1$, Firey \cite{Fi} defined what has become known
as the Minkowski-Firey $L_p$-combination $K+_p t\cdot L$ for $K, L
\in \mathcal{K}_{0}^{n+1}$ and $t\geq 0$ by letting
\begin{eqnarray*}
h^{p}_{K+_p t\cdot L}=h^{p}_{K}+th^{p}_{L}.
\end{eqnarray*}
This led to the notion of $k$-th $L_p$-surface area measure, $S_{p,
k}(K, \cdot)$, for each body $K \in \mathcal{K}_{0}^{n+1}$, via the
variational formula:
\begin{eqnarray*}
\frac{d}{dt}W_{k}(K+_p t\cdot
L)|_{t=0}=\frac{1}{n+1}\int_{\mathbb{S}^n}h_{L}^{p}(x)dS_{p, k}(K,
x),
\end{eqnarray*}
which holds for each  $L \in \mathcal{K}_{0}^{n+1}$, $p\geq 1$ and
$1\leq k\leq n$. Here $W_k(K)$ is the usual quermassintegral for
$K$. It was also shown in \cite{Lu93} that for each $K \in
\mathcal{K}_{0}^{n+1}$,
\begin{eqnarray*}
d S_{p, k}(K, \cdot)=h_{K}^{1-p}(x)dS_{k}(K, x),
\end{eqnarray*}
which shows that $k$-th $L_p$-surface area measure may be extended
to all $p \in \mathbb{R}$ in a completely obvious manner. Here
$dS_k(K, \cdot)$ is the $k$-th surface area measure of $K$.

The associated $L_p$-Christoffel-Minkowski problem in the
$L_p$-Brunn-Minkowski theory (first studied in \cite{Lu93}) asks:
For fixed $p \in \mathbb{R}$, given a Borel measure $\mu$ on
$\mathbb{S}^n$ (the data) what are necessary and sufficient
conditions on the measure $\mu$ to guarantee the existence of a body
$K \in \mathcal{K}_{0}^{n+1}$ such that $\mu=S_{p, k}(K, \cdot)$,
and if such a body $K$ exists to what extent is $K$ unique?

If the measure $\mu$ has a density function $\psi: \mathbb{S}^n\rightarrow [0, +\infty)$,
then the partial differential equation that is associated with the $L_p$-Christoffel-Minkowski
problem (with data $\psi$) is the $k$-Hessian type equation on $\mathbb{S}^n$
\begin{eqnarray}\label{GCM}
\sigma_{k}(u_{ij}+u\delta_{ij})=\psi(x)u^{p-1},
\end{eqnarray}
where $u_{ij}$ are the second order covariant derivatives with
respect to any orthonormal frame $\{e_1, e_2, ..., e_n \}$ on
$\mathbb{S}^n$, $\delta_{ij}$ is the standard Kronecker symbol and
\begin{equation*}
\sigma_k(u_{ij}+u\delta_{ij})=\sigma_k(\lambda(x)) = \sum _{1 \le
i_1 < i_2 <\cdots<i_k\leq
n}\lambda_{i_1}(x)\lambda_{i_2}(x)\cdots\lambda_{i_k}(x),
\end{equation*}
with $\lambda(x)=(\lambda_1(x), \lambda_2(x), \cdot\cdot\cdot,
\lambda_n(x))$ being the eigenvalues of the matrix
$u_{ij}+u\delta_{ij}$.

The $L_p$-Minkowski problem ($k=n$) has been extensively studied
during the last twenty years, see \cite{Bi17, Ch06, Lu93, Lu95,
Lu96} and see also \cite{Sch13} for the most comprehensive list of
results. When $p\geq1$, the existence and uniqueness of solutions
are well understood. However, when $p<1$ the uniqueness of solutions
to the $L_p$-Minkowski problem is very subtle, and indeed it was
shown in \cite{Jian15} that the uniqueness fails when $p<0$ even
restricted to smooth origin-symmetric convex bodies. Recently,
Brendle-Choi-Daskaspoulos' work \cite{Br17} implies the uniqueness
holds true for $1>p>-1-n$ and $\psi\equiv1$, and Chen-Huang-Li-Liu
\cite{Chen19} prove the uniqueness for $p$ close to $1$ and even
positive function $\psi$.

For $k<n$, $L_p$-Christoffel-Minkowski problem is difficult to deal
with, since the admissible solution to equation \eqref{GCM} is not
necessary a geometric solution to $L_p$-Christoffel- Minkowski
problem if $k<n$. So, one needs to deal with the convexity of the
solutions of \eqref{GCM}. Under a sufficient condition on the
prescribed function, Guan-Ma \cite{Guan03} proved the existence of a
unique convex solution. The key tool to handle the convexity is the
constant rank theorem for fully nonlinear partial differential
equations. Later, the equation \eqref{GCM} has been studied by
Hu-Ma-Shen \cite{Hu04} for $p\geq k+1$ and Guan-Xia \cite{Guan18}
for $1<p<k+1$ and even prescribed data, by using the constant rank
theorem. See also \cite{Guan10} for the proof of uniqueness and
\cite{Li15} for a simple proof. But for $p<1$, since the lack of
Firey's extension of Brunn-Minkowski inequality (See Corollary 1.3
in \cite{Lu93}) and constant rank theorem, the existence and
uniqueness are difficult and challenging problems. As far as I know,
the existence and uniqueness for $p<1$ are unknown until now. In
this paper, we make some progresses on the uniqueness for $p<1$ and
$\psi\equiv1$.

We consider the uniqueness of strictly convex solutions to the following
$L_p$-Christoffel-Minkowski problem:
\begin{eqnarray}\label{CM}
\sigma_{k}(u_{ij}+u\delta_{ij})=u^{p-1} \quad \mbox{on} \quad \mathbb{S}^n,
\end{eqnarray}
here the strict convexity of a solution, $u$,  means that the matrix
\begin{eqnarray*}
(u_{ij}+u\delta_{ij})
\end{eqnarray*}
is positive definite on $\mathbb{S}^n$.
We mainly get the following result.
\begin{theorem}\label{main}
Assume $u \in C^{4}(\mathbb{S}^n)$ is a strictly convex solution to \eqref{CM},
then $u\equiv constant$ for $1>p>1-k$.
\end{theorem}

\begin{remark}
We know from McCoy's work \cite{M11}, Theorem \ref{main} holds true
for $p=1-k$. Thus, Theorem \ref{main} in fact holds true for $1>
p\geq1-k$.
\end{remark}

Our proof is motivated by the idea of Choi-Daskaspoulos \cite{Ch16}
and Brendle-Choi-Daskaspoulos \cite{Br17} in which they show the
self-similar solution of $\alpha$-Gauss curvature flow, .i.e, an
embedded, strictly convex hypersurface in $\mathbb{R}^{n+1}$ given
by $X: \mathbb{S}^n\rightarrow \mathbb{R}^{n+1}$ satisfying the
equation
\begin{eqnarray*}
K^{\alpha}=\langle X, \nu\rangle
\end{eqnarray*}
is a sphere when $\alpha>\frac{1}{n+2}$, where $K$ and $\nu$ are the
Gauss curvature and out unit normal of $\Sigma$ respectively. Their
result is also equivalent to say that the $L_p$-Minkowski problem
\eqref{CM}($k=n$) has the unique solution $u\equiv 1$ for
$1>p>-n-1$. The idea of their proof is to apply Maximum Principles
for the following two important auxiliary functions which are
introduced in \cite{Ch16, Br17}:
\begin{eqnarray*}
W(x)=K^{\alpha}\lambda^{-1}_{1}(h_{ij})-\frac{n\alpha-1}{2n\alpha}|X|^2=u \cdot \lambda_1(b_{ij})
-\frac{n\alpha-1}{2n\alpha}(u^2+|Du|^2)
\end{eqnarray*}
and
\begin{eqnarray}\label{Old-Z}
Z(x)=K^{\alpha}\tr(b_{ij})-\frac{n\alpha-1}{2\alpha}|X|^2=u\cdot \tr(b_{ij})-\frac{n\alpha-1}{2\alpha}(u^2+|Du|^2),
\end{eqnarray}
where  $b_{ij}$ is the inverse matrix of the second fundamental form
$h_{ij}$ of $\Sigma$, $\lambda_1(b_{ij})$ is the biggest eigenvalues
of $b_{ij}$, $\lambda_{1}^{-1}(h_{ij})=\lambda_1(b_{ij})$, and $u:
\mathbb{S}^n\rightarrow \mathbb{R}$ is the support function of
$\Sigma$. Later, Gao-Li-Ma \cite{Gao18} and Gao-Ma\cite{Gao19} use
this two functions above to study the uniqueness of closed
self-similar solutions to $\sigma^{\alpha}_{k}$-curvature flow
following the idea of \cite{Ch16, Br17}. In details, in \cite{Gao18}
the authors consider the following general equation
\begin{eqnarray}\label{SEq}
S^{\alpha}(\kappa_1,...,\kappa_n)=\langle X, \nu\rangle,
\end{eqnarray}
where $S$ is a 1-homogeneous smooth symmetric function of the
principal curvatures $\kappa_i$ of the hypersurface $\Sigma$ given
by $X: \mathbb{S}^n\rightarrow \mathbb{R}^{n+1}$. Under some
assumptions on $S$, they show $\Sigma=X(\mathbb{S}^n)$ is a round
sphere for $\alpha\geq1$. Examples of $S$ include
$S=\sigma_{k}^{\frac{1}{k}}(\kappa_1,...,\kappa_n)$, but not include
$S=(\frac{\sigma_n}{\sigma_{n-k}})^{\frac{1}{k}}(\kappa_1,...,\kappa_n)$,
for which the equation \eqref{SEq} is equivalent to $L_p$
-Christoffel-Minkowski problem \eqref{CM} with
$p=1-\frac{1}{\alpha}$. The main difficulty lies in the
non-positivity of the term \eqref{Diff}
\begin{eqnarray*}
2\beta\bigg(k\sigma_k
f-n\sum_{i=1}^{n}\sigma_{k-1}(\lambda|i)\lambda^{2}_{i}\bigg),
\end{eqnarray*}
if $f=\sigma_1$. (In this case, the revised function $Z$
\eqref{New-Z} is just the original function $Z$ \eqref{Old-Z}.) To
overcome this difficulty, the easiest way is to choose $f$ such that
$$k\sigma_k
f-n\sum_{i=1}^{n}\sigma_{k-1}(\lambda|i)\lambda^{2}_{i}=0.$$ So, we
need to modify the function $Z$. We introduce the following two
auxiliary functions: one is the original function $W$
\begin{eqnarray*}
W(x)=u \cdot \lambda_1(b_{ij})-\beta(u^2+|Du|^2),
\end{eqnarray*}
the other is a new function $Z$
\begin{eqnarray}\label{New-Z}
Z(x)=u F(b_{ij})-n\beta(u^2+|Du|^2),
\end{eqnarray}
where $\lambda_n\leq...\leq \lambda_2 \leq \lambda_1$ are the eigenvalues
of the matrix $b_{ij}=u_{ij}+u\delta_{ij}$, $\beta=\frac{p-1+k}{2k}$ and
\begin{eqnarray*}
F(b_{ij})=f(\lambda_1, \lambda_2, ..., \lambda_n)
=\sum_{i=1}^{n}\frac{n\sigma_{k-1}(\lambda|i)\lambda^{2}_{i}}{k\sigma_k}=
\frac{n}{k}[\sigma_1-(k+1)\frac{\sigma_{k+1}}{\sigma_k}],
\end{eqnarray*}
here the third equality can be easily derived by Proposition
\ref{sigma}(7) and we denote by
$\sigma_{k-1}(\lambda|i)=\frac{\partial \sigma_k}{\partial
\lambda_i}$. To our surprise, we find that $f(\lambda_1, \lambda_2,
..., \lambda_n)$ is a convex function in the positive cone
$$\Gamma_{+}=\{(\lambda_1, ..., \lambda_n): \lambda_i>0 \quad
\mbox{for} \quad 1\leq i\leq n\},$$ since
$\frac{\sigma_{k+1}}{\sigma_k}$ is concave in $\Gamma_{+}$ by (1)
and (5) in Proposition \ref{sigma}. This fact is the key to our
proof.

\begin{remark}
We can propose the following questions:

(i) When $k=n$ our result does not cover the previous result in
\cite{Ch16, Br17}, then it is natural to ask if one can improve it?

(ii) Can we construct some non-uniqueness examples of solutions to
\eqref{CM} for $p<1-k$?
\end{remark}

\section{The proof of Theorem \ref{CM}}

Let $\lambda=(\lambda_1,\dots,\lambda_n)\in\mathbb{R}^n$, we recall
the definition of elementary symmetric function for $1\leq k\leq n$
\begin{equation*}
\sigma_k(\lambda)= \sum _{1 \le i_1 < i_2 <\cdots<i_k\leq
n}\lambda_{i_1}\lambda_{i_2}\cdots\lambda_{i_k}.
\end{equation*}
We also set $\sigma_0=1$ and $\sigma_k=0$ for $k>n$ or $k<0$. Recall
that the G{\aa}rding's cone is defined as
\begin{equation*}
\Gamma_k  = \{ \lambda  \in \mathbb{R}^n :\sigma _i (\lambda ) >
0,\forall 1 \le i \le k\}.
\end{equation*}
We denote by $\sigma_{k-1}(\lambda|i)=\frac{\partial
\sigma_k}{\partial \lambda_i}$ and
$\sigma_{k-2}(\lambda|ij)=\frac{\partial^2 \sigma_k}{\partial
\lambda_i\partial \lambda_j}$. Then, we list some properties of
$\sigma_k$ which will be used later.

\begin{proposition}\label{sigma}
Let $\lambda=(\lambda_1,\dots,\lambda_n)\in\mathbb{R}^n$ and $1\leq
k\leq n$, then we have

(1) $\Gamma_1\supset \Gamma_2\supset \cdot\cdot\cdot\supset
\Gamma_n=\Gamma_{+}$;

(2) $\sigma_{k-1}(\lambda|i)>0$ for $\lambda \in \Gamma_k$ and
$1\leq i\leq n$;

(3) $\sigma_k(\lambda)=\sigma_k(\lambda|i)
+\lambda_i\sigma_{k-1}(\lambda|i)$ for $1\leq i\leq n$;

(4) $\sum_{i=1}^{n}\frac{\partial(\frac{\sigma_{k+1}}{\sigma_{k}})}
{\partial \lambda_i}\geq \frac{n-k}{k+1}$ for $\lambda \in
\Gamma_{k+1}$;

(5) $\sigma^{\frac{1}{k}}_{k}$ and
$\Big[\frac{\sigma_k}{\sigma_l}\Big]^{\frac{1}{k-l}}$ are concave in
$\Gamma_k$ for $0\leq l<k$;

(6) If $\lambda_1\geq \lambda_2\geq \cdot\cdot\cdot\geq \lambda_n$,
then $\sigma_{k-1}(\lambda|1)\leq \sigma_{k-1}(\lambda|2)\leq
\cdot\cdot\cdot\leq \sigma_{k-1}(\lambda|n)$ for $\lambda \in
\Gamma_k$;

(7) $\sum_{i=1}^{n}\sigma_{k-1}(\lambda|i)\lambda^{2}_{i}
=\sigma_1\sigma_k-(k+1)\sigma_{k+1}$.
\end{proposition}

\begin{proof}
All the properties are well known. For example, see Chapter XV in
\cite{Li96} or \cite{Hui99} for proofs of (1), (2), (3),  (6) and
(7); see Lemma 2.2.19 in \cite{Ger06} for a proof of (4); see
\cite{CNS85} and \cite{Li96} for a proof of (5).
\end{proof}

We choose an orthonormal frame $\{e_1, e_2, ..., e_n \}$ on
$\mathbb{S}^n$. We use the notations $u_i=D_{e_i} u$,
$u_{ij}=D_{e_j}D_{e_i} u$, $D_{e_l}b_{ij}=b_{ij; l}$ ..., and so on,
where $D$ is the standard Levi-Civita connection on $\mathbb{S}^n$.
Set $b_{ij}=u_{ij}+u\delta_{ij}$, we denote by $\lambda_n(x)\leq
...\leq\lambda_2(x)\leq\lambda_1(x)$ are the eigenvalues of
$\{b_{ij}(x)\}$, arranged in decreasing order. Each eigenvalue
defines a Lipschitz continuous function on $\mathbb{S}^n$.

We recall the following Lemma which is similar to Lemma 5 in
\cite{Br17} and their proofs are almost the same.

\begin{lemma}\label{key}
Suppose that $\varphi$ is a smooth function on $\mathbb{S}^n$ such that
\begin{eqnarray*}
\lambda_1\leq \varphi \quad \mbox{everywhere and} \quad \lambda_1(x_0)=\varphi(x_0)
\end{eqnarray*}
for some $x_0 \in \mathbb{S}^n.$ Let $m$ denote the multiplicity of the biggest eigenvalue of $b_{ij}$ at $x_0$, so that
\begin{eqnarray*}
\lambda_n(x_0)\leq ... \leq \lambda_{m+1}(x_0)<\lambda_m(x_0)=...=\lambda_1(x_0).
\end{eqnarray*}
Then, we have
\begin{eqnarray}\label{b-1}
b_{ij; l}=\varphi_l\delta_{ij} \quad \mbox{at}\ x_0 \ \ \mbox{for} \ 1\leq i, j\leq m.
\end{eqnarray}
Moreover,
\begin{eqnarray*}
\varphi_{ii}\geq  b_{11; ii}+2\sum_{l>m}\frac{(b_{1l; i})^2}{\lambda_1-\lambda_l}, \quad \mbox{at} \ x_0.
\end{eqnarray*}
\end{lemma}

Now we list the following well-known result (See Lemma 3.2 in
\cite{Gao18} or \cite{And94}).

\begin{lemma}\label{Gao0}
If $W=(w_{ij})$ is a symmetric real matrix, $\lambda_i=\lambda_i(W)$
is one of its eigenvalues ($i = 1, ..., n$) and
$F=F(W)=f(\lambda(W))$ is a symmetric function of $\lambda_1, ...,
\lambda_n$, then for any real symmetric matrix $A= (a_{ij})$, we
have the following formulas:
\begin{eqnarray*}
\frac{\partial^2 F}{\partial w_{ij}\partial
w_{st}}a_{ij}a_{st}=\frac{\partial^2 f}{\partial\lambda_p
\partial\lambda_q}a_{pp}a_{qq}+2\sum_{p<q}\frac{\frac{\partial
f}{\partial \lambda_p}-\frac{\partial f}{\partial
\lambda_q}}{\lambda_p-\lambda_q}a^{2}_{pq}.
\end{eqnarray*}

\end{lemma}

We also need Lemma 4.4 in \cite{Gao18} which statement as follows.

\begin{lemma}\label{Gao1}
Under the assumptions of Lemma \ref{key}, we have at $x_0$
\begin{eqnarray*}
&&\sigma_{k}^{ij, pq}b_{ij; 1}b_{pq;1}-2\sigma_{k}^{ii}\sum_{l>m}\frac{(b_{1l; i})^2}{\lambda_1-\lambda_l}
\\&=&\sum_{i\neq j}\sigma_{k-2}(\lambda|ij)b_{ii; 1}b_{jj;1}-2\sum_{i>m}\sigma_{k-1}(\lambda|i)
\frac{(b_{11; i})^2}{\lambda_1-\lambda_i}-2\sum_{i>m}\sigma_{k-1}(\lambda|i)\frac{(b_{ii; 1})^2}
{\lambda_1-\lambda_i}\\&&+2\sum_{i>j>m}\frac{\sigma_{k-1}(\lambda|i)(\lambda_1-\lambda_i)^2
-\sigma_{k-1}(\lambda|j)(\lambda_1-\lambda_j)^2}{(\lambda_1-\lambda_i)(\lambda_1-\lambda_j)(\lambda_i-\lambda_j)}b^{2}_{ij; 1}.
\end{eqnarray*}
\end{lemma}

\begin{proof}
For convenience, we give a proof here. We know from Lemma \ref{Gao0}
\begin{eqnarray*}
\sigma_{k}^{ij, pq}b_{ij; 1}b_{pq;1}=\sum_{i\neq
j}\sigma_{k-2}(\lambda|ij)b_{ii; 1}b_{jj;1}+2\sum_{i>j}
\bigg(\sigma_{k-1}(\lambda|i)-\sigma_{k-1}(\lambda|j)\bigg)\frac{(b_{ij;
1})^2}{\lambda_i-\lambda_j}.
\end{eqnarray*}
Since we have by \eqref{b-1}
\begin{eqnarray*}
b_{ij;1}=b_{1i; j}=0
\end{eqnarray*}
for $2\leq i\leq m$, we can obtain
\begin{eqnarray*}
\sigma_{k}^{ij, pq}b_{ij; 1}b_{pq;1}&=&\sum_{i\neq
j}\sigma_{k-2}(\lambda|i j)b_{ii;
1}b_{jj;1}+2\sum_{i>m}\bigg(\sigma_{k-1}(\lambda|i)
-\sigma_{k-1}(\lambda|1)\bigg)\frac{(b_{11;
i})^2}{\lambda_i-\lambda_1}\\&&+2\sum_{i>j>m}
\bigg(\sigma_{k-1}(\lambda|i)-\sigma_{k-1}(\lambda|j)\bigg)\frac{(b_{ij;
1})^2}{\lambda_i-\lambda_j}
\end{eqnarray*}
and
\begin{eqnarray*}
2\sigma_{k}^{ii}\sum_{l>m}\frac{(b_{1l; i})^2}{\lambda_1-\lambda_l}&=&2\sigma_{k-1}(\lambda|1)\sum_{l>m}\frac{(b_{11; l})^2}{\lambda_1-\lambda_l}+2\sum_{i>m}\sigma_{k-1}(\lambda|i)\frac{(b_{1i; i})^2}{\lambda_1-\lambda_i}\\&&+2\sum_{i>l>m}\sigma_{k-1}(\lambda|i)\frac{(b_{1l; i})^2}{\lambda_1-\lambda_l}+2\sum_{l>i>m}\sigma_{k-1}(\lambda|i)\frac{(b_{1l; i})^2}{\lambda_1-\lambda_l},
\end{eqnarray*}
the lemma follows by adding the above two equations together.
\end{proof}

Now, we begin to prove Theorem \ref{main}. Set
\begin{eqnarray*}
W(x)=u \cdot \lambda_1(x)-\beta(u^2+|Du|^2),
\end{eqnarray*}
where $\beta=\frac{p-1+k}{2k}$ and $\lambda_n(x)\leq...\leq\lambda_1(x)$ are the eigenvalues of $\{b_{ij}(x)\}$. Since $1>p>1-k$, so $0<\beta<\frac{1}{2}$. Our proof is divided into two steps.

\

\textbf{Step 1}: we will prove
\begin{eqnarray*}
\lambda_1(x_0)=\lambda_2(x_0)=...=\lambda_n(x_0) \quad \mbox{and} \quad |D u|(x_0)=0
\end{eqnarray*}
for any $x_0 \in \{x \in \mathbb{S}^n: W(x)=\max_{\mathbb{S}^n}W\}$.

\

Assume $W(x)$ attains its maximum at $x_0$. As above, we denote by $m$
the multiplicity of the biggest eigenvalue of $b_{ij}$ at $x_0$. Let us define a
smooth function $\varphi$ such that
\begin{eqnarray*}
W(x_0)=u \cdot \varphi-\beta(u^2+|Du|^2).
\end{eqnarray*}
Since $W$ attains its maximum at $x_0$, we
have $\varphi(x)\geq \lambda_1(x)$ everywhere and $\lambda_1=\varphi$ at $x_0$.
Choose an orthonormal frame at $x_0$ such that
\begin{eqnarray*}
b_{ij}(x_0)=diag\{\lambda_1(x_0), ..., \lambda_n(x_0)\}
\end{eqnarray*}
with
\begin{eqnarray*}
\lambda_n(x_0)\leq ... \leq \lambda_{m+1}(x_0)<\lambda_m(x_0)=...=\lambda_1(x_0).
\end{eqnarray*}
Since $u \cdot \varphi-\beta(u^2+|Du|^2)=constant$, then
\begin{eqnarray}\label{1-diff}
[u\cdot\varphi-\beta(u^2+|Du|^2)]_i=0
\end{eqnarray}
and
\begin{eqnarray}\label{2-diff}
[u\cdot\varphi-\beta(u^2+|Du|^2)]_{ii}=0.
\end{eqnarray}
Taking \eqref{1-diff}'s value at $x_0$, we have by \eqref{b-1}
\begin{eqnarray*}
0=u_i\lambda_1+u b_{11; i}-2\beta u_i\lambda_i.
\end{eqnarray*}
Thus,
\begin{eqnarray*}
b_{11; i}=\frac{u_i}{u}(2\beta \lambda_i-\lambda_1),
\end{eqnarray*}
which implies together with \eqref{b-1} and $0<2\beta<1$, $u_i(x_0)=0$ for $2\leq
i\leq m$. Taking \eqref{2-diff}'s value at $x_0$ results in
\begin{eqnarray*}
0=u_{ii}\lambda_1+2u_ib_{11; i}+u
\varphi_{ii}-2\beta[u^{2}_{i}+uu_{ii}+u_{li}^{2}+u_lu_{lii}]
\end{eqnarray*}
Thus, we obtain at $x_0$ using Lemma \ref{key}
\begin{eqnarray*}
0&=&u_{ii}\lambda_1+2u_ib_{11; i}+u
\varphi_{ii}-2\beta[u^{2}_{i}+uu_{ii}+u_{li}^{2}+u_lu_{lii}]\\&\geq&(\lambda_i-u)\lambda_1+2u_i\frac{u_i}{u}(2\beta
\lambda_i-\lambda_1)+u \bigg(b_{11; ii}+2\sum_{l>m}\frac{(b_{1l;
i})^2}{\lambda_1-\lambda_l}\bigg)\\&&-2\beta[\lambda_i(\lambda_i-u)+u_lb_{ii;
l}]\\&\geq&\lambda_i\lambda_1-u\lambda_i+\frac{2u^{2}_{i}}{u}(2\beta
\lambda_i-\lambda_1)+u \bigg(b_{ii; 11}+2\sum_{l>m}\frac{(b_{1l;
i})^2}{\lambda_1-\lambda_l}\delta_{1i}\bigg)\\&&-2\beta[\lambda^{2}_{i}-u\lambda_i+u_lb_{ii;
l}],
\end{eqnarray*}
here we use the following Ricci identity (see \cite{Ham82} or (1.30)
in \cite{Chow06})
\begin{eqnarray*}
b_{ii;11}=b_{11;ii}-b_{11}+b_{ii}
\end{eqnarray*}
to get the last inequality.
Differentiating the equation \eqref{CM} shows
\begin{eqnarray*}
\sigma_{k}^{ii}b_{ii;1}=(p-1)u^{p-2}u_1.
\end{eqnarray*}
Differentiating it again
\begin{eqnarray*}
\sigma_{k}^{ii}b_{ii;11}&=&(p-1)u^{p-3}[uu_{11}+(p-2)u^{2}_{1}]-\sigma_{k}^{ij, st}b_{ij; 1}b_{st; 1}\\&=&(p-1)u^{p-3}[u\lambda_1-u^2+(p-2)u^{2}_{1}]-\sigma_{k}^{ij, st}b_{ij; 1}b_{st;1}.
\end{eqnarray*}
Due to the concavity of $\sigma^{\frac{1}{k}}_{k}(\lambda)$ (see (5)
in Proposition \ref{sigma}),
\begin{eqnarray*}
-\sum_{i\neq j}\sigma_{k-2}(\lambda|i j)b_{ii; 1}b_{jj;1}\geq-\frac{(k-1)(p-1)^2}{k}u^{p-3}u^{2}_{1},
\end{eqnarray*}
which results in together with Lemma \ref{Gao1} by noticing that the
forth and fifth terms in the right hand of the equation in Lemma
\ref{Gao1} are negative (this fact can be easily seen by Proposition
\ref{sigma} (6))
\begin{eqnarray*}
&&-\sigma_{k}^{ij, st}b_{ij; 1}b_{st;1}+2\sigma_{k}^{ii}\sum_{l>m}\frac{(b_{1l; i})^2}
{\lambda_1-\lambda_l}\\&\geq&-\frac{(k-1)(p-1)^2}{k}u^{p-3}u^{2}_{1}+2\sum_{i>m}\sigma_{k-1}(\lambda|i)\frac{(b_{11; i})^2}{\lambda_1-\lambda_i}.
\end{eqnarray*}
Then, we arrive at $x_0$ by noticing that $u_i(x_0)=0$ for $2\leq i\leq m$
\begin{eqnarray*}
0&=&\sigma_{k}^{ij}[u \cdot \varphi-\beta(u^2+|Du|^2)]_{ij}
\\&\geq&ku^{p-1}\lambda_1-ku^p+\sum_{i=1}^{n}\sigma_{k-1}(\lambda|i)\frac{2u^{2}_{i}}{u}(2\beta
\lambda_i-\lambda_1)\\&&+u
\bigg((p-1)u^{p-3}[u\lambda_{1}-u^2+(p-2)u^{2}_{1}]-\frac{(k-1)(p-1)^2}{k}u^{p-3}u^{2}_{1}\\&&+\frac{2}{u^2}\sum_{i>
m}\frac{\sigma_{k-1}(\lambda|i)}{\lambda_1-\lambda_i}u^{2}_{i}(2\beta\lambda_i-\lambda_1)^2\bigg)
\\&&-2\beta\bigg(\sum_{i=1}^{n}\sigma_{k-1}(\lambda|i)\lambda_{i}(\lambda_i-\lambda_1)+ku^{p-1}
\lambda_1-ku^p+(p-1)u^{p-2}\sum_{i=1}^{n}u^{2}_{i}\bigg)\\&\geq& u^{p-1}(k\lambda_1+(p-1)\lambda_1-2\beta k\lambda_1)\\&&+u^p(-k-(p-1)+2k\beta)+2\sum_{i=1}^{n}\beta\sigma_{k-1}(\lambda|i)\lambda_i(\lambda_1-\lambda_i)
\\&&+\frac{u^{2}_{1}}{u}\bigg(2(2\beta-1)\sigma_{k-1}(\lambda|1)\lambda_1+(p-1)(p-2)u^{p-1}\\&&-
\frac{(k-1)(p-1)^2}{k}u^{p-1}-2\beta(p-1)u^{p-1}\bigg)
\\&&+\frac{2}{u}\sum_{i> m}\sigma_{k-1}(\lambda|i)u^{2}_{i}(2\beta\lambda_i-\lambda_1)\frac{(2\beta-1)\lambda_i}{\lambda_1-\lambda_i}\\&&+
2\beta(1-p)u^{p-2}\sum_{i>m}u^{2}_{i}+2\sum_{i=1}^{n}\beta\sigma_{k-1}(\lambda|i)\lambda_i(\lambda_1-\lambda_i)
\\&\geq&u^{2}_{1}u^{p-2}\frac{2(k-1)(1-p)}{k}
+\frac{2}{u}\sum_{i>
m}\sigma_{k-1}(\lambda|i)u^{2}_{i}(2\beta\lambda_i-\lambda_1)\frac{(2\beta-1)\lambda_i}{\lambda_1-\lambda_i}\\&&+
2\beta(1-p)u^{p-2}\sum_{i>m}u^{2}_{i}+2\sum_{i=1}^{n}\beta\sigma_{k-1}(\lambda|i)\lambda_i(\lambda_1-\lambda_i)
\\&\geq& 0,
\end{eqnarray*}
where we use the following inequality to get the last but one inequality
\begin{eqnarray*}
\sigma_{k}(\lambda)\geq\lambda_1\sigma_{k-1}(\lambda|1),
\end{eqnarray*}
which can be easily proved in view of the assumption on the positive
definite of $u_{ij}+u\delta_{ij}$ and Proposition \ref{sigma}
(1)(2)(3). Thus, $\lambda_1(x_0)=\lambda_2(x_0)=...=\lambda_n(x_0)$
and $|D u|(x_0)=0$.

\

\textbf{Step 2}:  we want to show that $\{x \in \mathbb{S}^n: W(x)=\max_{\mathbb{S}^n}W\}$ is an open set.

\
We define
\begin{eqnarray*}
Z(x)=u F(b_{ij})-n\beta(u^2+|Du|^2),
\end{eqnarray*}
where
\begin{eqnarray*}
F(b_{ij})=f(\lambda_1, \lambda_2, ..., \lambda_n)
=\sum_{i=1}^{n}\frac{n\sigma_{k-1}(\lambda|i)\lambda^{2}_{i}}{k\sigma_k}=
\frac{n}{k}[\sigma_1-(k+1)\frac{\sigma_{k+1}}{\sigma_k}],
\end{eqnarray*}
here the third equality is derived by Proposition \ref{sigma}(7).
Clearly, $f(\lambda_1, \lambda_2, ..., \lambda_n)$ is a
1-homogeneous convex function by Proposition \ref{sigma} (1)(5) and
satisfies $f(1, 1, ..., 1)=n$ , since $u$ is strictly convex. We
will prove for any $x_0 \in \{x \in \mathbb{S}^n:
W(x)=\max_{\mathbb{S}^n}W\}$, there exists a small neighborhood
$U(x_0)$ of $x_0$ such that
\begin{eqnarray*}
\sigma_{k}^{ij}Z_{ij}-\frac{2}{u}\sigma_{k}^{ij}u_i Z_j\geq 0
\end{eqnarray*}
and
\begin{eqnarray*}
Z(x_0)=\max_{U(x_0)}Z(x).
\end{eqnarray*}
Denoting by $f_i=\frac{\partial f}{\partial \lambda_i}$ and $f_{ij}=\frac{\partial^2 f}{\partial \lambda_i\partial \lambda_j}$.
Since $f_i(\lambda(x_0))=1$, we can choose $U(x_0)$ such that $f_i$ is positive in $U(x_0)$.
For any $x \in U(x_0)$, we choose a coordinate at $x$ such that
\begin{eqnarray*}
b_{ij}(x)=diag\{\lambda_1(x), ..., \lambda_n(x)\}.
\end{eqnarray*}
Then, we have at $x$
\begin{eqnarray*}
Z_{i}&=&u_{i} f+uF^{jl}b_{jl; i}-2n\beta u_i\lambda_i
\end{eqnarray*}
and
\begin{eqnarray*}
Z_{ii}&=&u_{ii}f+2u_i\sum_{l=1}^{n}f_lb_{ll; i}+uF^{jl, st}b_{jl; i}b_{st; i}+u \sum_{l=1}^{n}f_lb_{ll;ii}-2n\beta[u^{2}_{i}+uu_{ii}+u_{li}^{2}+u_lu_{lii}]
\\&=&\lambda_{i}f-u f+2u_i\bigg(\frac{Z_i}{u}+\frac{u_i}{u}(2n\beta\lambda_i-f)\bigg)+uF^{jl, st}b_{jl; i}b_{st; i}+u \sum_{l=1}^{n}f_lb_{ll;ii}\\&&-2n\beta[\lambda^{2}_{i}-u\lambda_{i}+u_{l}b_{ii;l}]
\\&=&\lambda_{i}f+u\bigg(-f+\sum_{l}f_l \lambda_{l}-\lambda_i \sum_{l}f_l\bigg)+2u_i\bigg(\frac{Z_i}{u}+\frac{u_i}{u}(2n\beta\lambda_i-f)\bigg)
\\&&+uF^{jl, st}b_{jl; i}b_{st; i}+u \sum_{l=1}^{n}f_l b_{ii;ll}-2n\beta[\lambda^{2}_{i}-u\lambda_{i}+\sum_{l=1}^{n}u_{l}b_{ii;l}]
\\&=&\lambda_{i}f-u\lambda_i \sum_{l=1}^{n}f_l+2u_i\bigg(\frac{Z_i}{u}+\frac{u_i}{u}(2n\beta\lambda_i-f)\bigg)
\\&&+uF^{jl, st}b_{jl; i}b_{st; i}
+u \sum_{l=1}^{n}f_l b_{ii;ll}-2n\beta[\lambda^{2}_{i}-u\lambda_{i}+\sum_{l=1}^{n}u_{l}b_{ii;l}]\\&\geq&\lambda_{i}f-u\lambda_i \sum_{l=1}^{n}f_l+2u_i\bigg(\frac{Z_i}{u}+\frac{u_i}{u}(2n\beta\lambda_i-f)\bigg)
\\&&+u \sum_{l=1}^{n}f_l b_{ii;ll}-2n\beta[\lambda^{2}_{i}-u\lambda_{i}+\sum_{l=1}^{n}u_{l}b_{ii;l}]
\end{eqnarray*}
in view of the Ricci identity (see \cite{Ham82} or (1.30) in
\cite{Chow06})
$$b_{ii;ll}=b_{ll;ii}-b_{ll}+b_{ii},$$ and we use the convexity of $f$ to get the last inequality.
Differentiating the equation \eqref{CM} shows

\begin{eqnarray*}
\sigma_{k}^{ii}b_{ii;l}=(p-1)u^{p-2}u_l.
\end{eqnarray*}
Differentiating it again
\begin{eqnarray*}
\sigma_{k}^{ii}b_{ii;ll}&=&(p-1)u^{p-3}[uu_{ll}+(p-2)u^{2}_{l}]-\sigma_{k}^{ij, pq}b_{ij; l}b_{pq;l}\\&=&(p-1)u^{p-3}[u\lambda_l-u^2+(p-2)u^{2}_{l}]-\sigma_{k}^{ij, pq}b_{ij; l}b_{pq;l}.
\end{eqnarray*}
Due to the concavity of $\sigma^{\frac{1}{k}}_{k}$ (see (5) in
Proposition \ref{sigma}),
\begin{eqnarray*}
\sigma_{k}^{ij, pq}b_{ij; l}b_{pq;l} \leq
\frac{k-1}{k}\frac{(\sigma_{k}^{ij}b_{ij; l})^2}{\sigma_k}
\end{eqnarray*}
and the positivity of $f_i$ in $U(x_0)$, we arrive
\begin{eqnarray}\label{Diff}
&& \sigma_{k}^{ij}Z_{ij}
\nonumber\\&\geq&ku^{p-1}f-ku^p\sum_{i=1}^{n}f_i+\frac{2}{u}\sigma_{k}^{ij}u_i Z_j
+\sum_{i=1}^{n}\sigma_{k-1}(\lambda|i)\frac{2u^{2}_{i}}{u}(2n\beta
\lambda_i-f)
\nonumber\\&&+u \bigg((p-1)u^{p-3}[u f-u^2\sum_{i=1}^{n}f_i+(p-2)\sum_{i=1}^{n}u^{2}_{i}f_i]
-\frac{k-1}{k}(p-1)^2u^{p-3}\sum_{i=1}^{n}u^{2}_{i}f_i\bigg)\nonumber\\&&-2n\beta\bigg(\sum_{i=1}^{n}\sigma_{k-1}(\lambda|i)
\lambda^{2}_{i}-ku^p+(p-1)u^{p-2}\sum_{i=1}^{n}u^{2}_{i}\bigg)
\nonumber\\&\geq& u^{p-1}(k f+(p-1)f-2\beta k f)
+u^p\bigg(-k\sum_{i=1}^{n}f_i-(p-1)\sum_{i=1}^{n}f_i+2n k\beta\bigg)+\frac{2}{u}\sigma_{k}^{ij}u_i
Z_j\nonumber\\&&+\sum_{i=1}^{n}u^{2}_{i}u^{p-2} \bigg(2\sigma_{k-1}(\lambda|i)(2n\beta
\lambda_i-f)u^{1-p}+(p-1)(p-2)f_i\nonumber\\&&-\frac{k-1}{k}(p-1)^2f_i+2n\beta(1-p)\bigg)
\nonumber\\&&+2\beta\bigg(k\sigma_k
f-n\sum_{i=1}^{n}\sigma_{k-1}(\lambda|i)\lambda^{2}_{i}\bigg)
\\&\geq&u^p\bigg(-k\sum_{i=1}^{n}f_i-(p-1)\sum_{i=1}^{n}f_i+2n k\beta\bigg)+\frac{2}{u}\sigma_{k}^{ij}u_i
Z_j\nonumber\\&&+\sum_{i=1}^{n}u^{2}_{i}u^{p-2} \bigg(2\sigma_{k-1}(\lambda|i)(2n\beta
\lambda_i-f)u^{1-p}+(p-1)(p-2)f_i\nonumber\\ \nonumber&&-\frac{k-1}{k}(p-1)^2f_i+2n\beta(1-p)\bigg).
\end{eqnarray}
At $x_0$, we have $\lambda_1=\lambda_2=...=\lambda_n$. Thus, at $x_0$
\begin{eqnarray*}
n\lambda_i=f \quad \mbox{and} \quad f_{i}(x_0)=1 \quad \forall 1\leq i\leq n.
\end{eqnarray*}
Thus,
\begin{eqnarray*}
2\sigma_{k-1}(\lambda|i)(2n\beta
\lambda_i-f)=2k(2\beta-1)u^{p-1}.
\end{eqnarray*}
So,
\begin{eqnarray*}
&&2\sigma_{k-1}(\lambda|i)(2n\beta
\lambda_i-f)u^{1-p}+(p-1)(p-2)f_i-\frac{k-1}{k}(p-1)^2f_i+2n\beta(1-p)
\\&=&\frac{(1-p)(n-1)(p+k-1)}{k}>0
\end{eqnarray*}
for $1>p>1-k$. Thus, there exists a small neighborhood $U(x_0)$ such that
\begin{eqnarray}\label{U-1}
&&u^{2}_{i}u^{p-2} \bigg(2\sigma_{k-1}(\lambda|i)(2n\beta
\lambda_i-f)u^{1-p}+(p-1)(p-2)f_i\nonumber\\&&-\frac{k-1}{k}(p-1)^2f_i+2n\beta(1-p)\bigg)
> 0.
\end{eqnarray}
Moreover, we obtain by (4) in Proposition \ref{sigma}
\begin{eqnarray*}
\sum_{i=1}^{n}\frac{\partial(\frac{\sigma_{k+1}}{\sigma_{k}})}
{\partial \lambda_i}\geq \frac{n-k}{k+1},
\end{eqnarray*}
which implies
\begin{eqnarray*}
\sum_{i=1}^{n}f_i\leq \frac{n}{k}[n-(k+1)\frac{n-k}{k+1}]=n.
\end{eqnarray*}
Thus, we have by $p>1-k$
\begin{eqnarray}\label{U-3}
u^p(-k\sum_{l}f_l-(p-1)\sum_{l}f_l+2n k\beta)=u^p(k+p-1)(n-\sum_{l}f_l)\geq 0.
\end{eqnarray}
Thus, combining \eqref{U-1} and \eqref{U-3}, we can find $U(x_0)$
such that
\begin{eqnarray*}
\sigma_{k}^{ij}Z_{ij}-\frac{2}{u}\sigma_{k}^{ij}u_i Z_j
\geq0.
\end{eqnarray*}
Since $f_i(\lambda_0)=1$, we can choose $U(\lambda_0)$ such that $f$
is increasing with each $\lambda_i$ in $U(\lambda_0)$, where
$\lambda_0=(\lambda_1(x_0), ..., \lambda_n(x_0))$. Then, we can
choose $U(x_0)$ such that $\{\lambda(x): x \in U(x_0)\}\subset
U(\lambda_0)$. So, we have
\begin{eqnarray*}
Z(x_0)=nW(x_0)\geq n W(x)\geq Z(x),
\end{eqnarray*}
which implies
\begin{eqnarray*}
Z(x_0)=\max_{U(x_0)}Z(x).
\end{eqnarray*}
Thus, we have by the strong maximum principle
\begin{eqnarray*}
Z(x_0)=Z(x)
\end{eqnarray*}
for any $x$ in $U(x_0)$, which implies
\begin{eqnarray*}
W(x_0)=W(x)
\end{eqnarray*}
for any $x$ in $U(x_0)$.
Thus, $W(x)\equiv \max_{\mathbb{S}^n}W$ for any $x \in \mathbb{S}^n$. So, $Du=0$ which implies $u=constant$.
Thus, we complete our proof.

\textbf{Acknowledgement:} Parts of this work were done, while the
author was visiting the mathematical institute of
Albert-Ludwigs-Universit\"{a}t Freiburg in Germany. He would like to
express his deep gratitude to Prof. Guofang Wang for invitation,
continuous support, encouragement, and some important suggestions on
this paper. He also thanks the mathematical institute of
Albert-Ludwigs-Universit\"{a}t Freiburg for its hospitality.
Moreover, he also thanks Prof. Chuanqiang Chen for some suggestions
on this paper. Lastly, he is grateful to the reviewer's valuable
comments, which is helpful for improving the manuscript.


\vspace{0.5 cm}




\vspace {1cm}

\begin{thebibliography}{50}
\setlength{\itemsep}{-0pt} \small

\bibitem{And94} B. Andrews, Contraction of convex hypersurfaces in Euclidean
space, Calc. Var. Partial Differ. Equ. 2 (1994) 151-171.

\bibitem{Bi17} G. Bianchi, K. J. B\"or\"oczky, A. Colesanti, Smoothness in the
$L_p$-Minkowski problem for $p<1$, J. Geom. Anal. 30 (1) (2020)
680-705.

\bibitem{Br17} S. Brendle, K. Choi, P. Daskalopoulos,
Asymptotic behavior of flows by powers of the Gaussian curvature,
Acta Mathematica 219 (2017) 1-16.

\bibitem{CNS85} L. Caffarelli, L. Nirenberg, J. Spruck, The Dirichlet problem for
nonlinear second order elliptic equations, III: Functions of the
eigenvalues of the Hessian, Acta Math. 155 (1985) 261-301.

\bibitem{Chen19} S.B. Chen, Y. Huang, Q.R. Li, J.K. Liu, $L_p$-Brunn-Minkowski inequality for
$p \in \Big(1-\frac{c}{n^{\frac{3}{2}}}, 1\Big)$, preprint,
arXiv:1811.10181, 2018.

\bibitem{Ch16} K. Choi, P. Daskalopoulos, Uniqueness of closed self-similar solutions to the
Gauss curvature flow, preprint, arXiv:1609.05487, 2016.

\bibitem{Ch06} K.S. Chou, X.J. Wang, The $L_p$-Minkowski problem and the
Minkowski problem in centroaffine geometry, Adv. in Math. 205 (2006)
33-83.

\bibitem{Chow06} B. Chow, P. Lu, L. Ni, Hamilton's Ricci flow.
Lectures in Contemporary Mathematics,3, Science Press and Graduate
Studies in Mathematics, 77, American Mathematical Socitety
(co-publication), 2006.

\bibitem{Fi} W. Firey, $p$-means of convex bodies, Math. Scand. 10 (1962)
17-24.

\bibitem{Gao19} S.Z. Gao, H. Ma, Self-similar solutions of curvature flows
in warped products, Differential Geom. Appl. 62 (2019) 234-252.

\bibitem{Gao18} S.Z. Gao, H.Z. Li, H. Ma, Uniqueness of closed self-similar
solutions to $\sigma^{\alpha}_{k}$-curvature flow, NoDEA Nonlinear
Differential Equations Appl. 5 (2018) Paper No.45, 26 pp.

\bibitem{Ger06} C. Gerhardt, Curvature problems, Series in Geometry and Topology, vol. 39,
International Press of Boston Inc., Sommerville, 2006.

\bibitem{Guan03} P.F. Guan, X.N. Ma, Christoffel-Minkowski problem I: convexity of
solutions of a hessian equation, Invent. Math. 151, (2003) 553-577.

\bibitem{Guan18} P.F. Guan, C. Xia, $L_p$ Christoffel-Minkowski problem: the case $1<p<
k+1$, Cal. Var. Partial Differential Equations, 2 (2018) 57-69.

\bibitem{Guan10} P.F. Guan, X.N. Ma, N. Trudinger, X.H. Zhu,
A Form of Alexandrov-Fenchel Inequality, Pure and Applied
Mathematics Quarterly Volume 6, Number 4 (Special Issue: In honor of
Joseph J. Kohn, Part 2 of 2) 999-1012, 2010.

\bibitem{Ham82} R. Hamilton, Three-manifolds with positive Ricci curvature, J. Differential
Geom. 17 (2) (1982) 255-306.

\bibitem{Hui99} G. Huisken, C. Sinestrari, Convexity estimates for mean
curvature flow and singularities of mean convex surfaces, Acta Math.
183 (1) (1999) 45-70.

\bibitem{Li96} G. Lieberman, Second order parabolic differential equations. World Scientific, 1996.

\bibitem{Jian15} H.Y. Jian, J. Lu, X.J. Wang, Nonuniqueness of solutions to the $L_p$-Minkowski problem, Adv. Math.
281 (2015) 845-856.

\bibitem{Hu04} C.Q. Hu, X.N. Ma, C.L. Shen, On the Christoffel-Minkowski problem of
Firey's $p$-sum, Cal. Var. Partial Differential Equations 21 (2004)
137-155 .

\bibitem{Li15} S.Y. Li, Christoffel-Minkowski problem of Firey's $p$-sum,
Master thesis, Chinese Academy of Sciences (2015)

\bibitem{Lu93} E. Lutwak, The Brunn-Minkowski-Firey theory I: Mixed volumes and the
Minkowski problem, J. Differential Geom. 38 (1993) 131-150.

\bibitem{Lu96} E. Lutwak, The Brunn-Minkowski-Firey theory II: Affine and
geominimal surface areas, Adv. in Math. 118 (1996) 244-294.

\bibitem{Lu95} E. Lutwak, V. Oliker, On the regularity of solutions to a
generalization of the Minkowski problem, J. Differential Geom. 41
(1995) 227-246.

\bibitem{M11} J. McCoy, Self-similar solutions of fully nonlinear curvature
flows, Ann. Sc. Norm. Super. Pisa Cl. Sci.(5) 10 (2) (2011) 317-333.

\bibitem{Sch13} R. Schneider, Convex bodies: the Brunn-Minkowski theory, Second
edition, No.151. Cambridge University Press, 2013.

\end{thebibliography}
\end{document}